\documentclass[a4paper,12pt]{amsart}

\usepackage{a4wide}
\usepackage{pgf,tikz}
\usepackage{amssymb}
\allowdisplaybreaks
\usepackage{enumerate}
	
\usepackage{amsmath}

\newif\ifdetails
\detailstrue
\newcommand{\DETAIL}[1]%
{\ifdetails\par\fbox{\begin{minipage}{0.9\linewidth}\textit{Detail:}
      #1\end{minipage}}\par\fi}
\newcommand{\TODO}[1]%
{\ifdetails\par\fbox{\begin{minipage}{0.9\linewidth}\textbf{TODO:}
      #1\end{minipage}}\par\fi}

\usepackage{makecell}
\usepackage{relsize}

\usepackage[pagewise,displaymath, mathlines]{lineno}
\newtheorem{thm}{Theorem}
\newtheorem{lem}{Lemma}

\newtheorem{cor}{Corollary}

\newtheorem{prop}{Proposition}

\usepackage{caption}
\usepackage{subcaption}
\usepackage{float} % provides H as float placement specifier
\usepackage[pdftex,a4paper,
citecolor = blue, colorlinks=true,urlcolor=blue]{hyperref}
\usepackage{mathrsfs}
\usetikzlibrary{arrows}

\usepackage{xcolor}

\newcommand{\old}[1]{{}}
\title[DOSSOU-OLORY, ANDRIANTIANA, RAKOTONARIVO, AND SHOZI]{Ratio of the number of $\mathbf{1}$-nearly independent vertex subsets and the Merrifield-Simmons index}

\author[Dossou-Olory]{Audace A. V. Dossou-Olory}
\author[Andriantiana]{Eric O. D.  Andriantiana}
\author[Rakotonarivo]{Valisoa R. M. Rakotonarivo}
\author[Shozi]{Zekhaya B. Shozi}

\thanks{}

\address{Audace A. V. Dossou-Olory \\  Institut National de l'Eau, Abomey-Calavi \\ and  Institut de Math\'ematiques et de Sciences Physiques, Dangbo \\ Universit\'e d'Abomey-Calavi, B\'enin }
\email{audace@aims.ac.za \quad or \quad audace.dossouolory@uac.bj}

\address{Eric O. D. Andriantiana \\ Department of Mathematics (Pure and Applied) \\ Rhodes University \\ Makhanda, 6140 South Africa \\ and National Institute for Theoretical and Computational Sciences (NITheCS), Stellenbosch, South Africa}
\email{e.andriantiana@ru.ac.za}

\address{Valisoa R. M. Rakotonarivo\\ Department of Mathematics and Applied Mathematics \\ University of Pretoria\\ Hatfield, 0002 South Africa\\ and National Institute for Theoretical and Computational Sciences (NITheCS), Stellenbosch, South Africa}
\email{valisoa.rakotonarivo@up.ac.za}

\address{Zekhaya B. Shozi \\ Discipline of Mathematics \\ University of KwaZulu-Natal \\ Durban, 4000 South Africa and National Institute for Theoretical and Computational Sciences (NITheCS), Stellenbosch, South Africa}
\email{shoziz1@ukzn.ac.za  \quad or \quad zekhaya@aims.ac.za}


\begin{document}

\begin{abstract}
The number $\sigma_k(G)$ of induced subgraphs with size $k$ of a graph $G$ was introduced recently as the number of $k$-nearly independent vertex subsets of $G$. Results highlighting similarity and difference in the behaviours of $\sigma_1$ and $\sigma_0$, have been reported. In this paper, we provide more comparison tools, by studying the ratio $\frac{\sigma_1(G)}{\sigma_{0}(G)}$. We establish sharp lower and upper bounds for this ratio over various classes of graphs, including connected graphs, trees, and forests. 

\end{abstract}

\maketitle

\section{Introduction}
The study of independent sets and their generalisations is a popular topic in graph theory. It is motivated by connections to combinatorics, statistical physics, and chemical graph theory. Given a graph $G$, an \emph{independent vertex subset} is a set of vertices, no two of which are adjacent in $G$. The number of such sets in $G$, denoted by $\sigma_0(G)$, is often known as the Merrifield-Simmons index of $G$. It has been extensively studied since its introduction by Merrifield and Simmons~\cite{merrifield1981enumeration} in the context of molecular structure analysis. This invariant, together with related quantities such as the Hosoya index and graph energy, has generated a rich body of extremal and structural results; see the survey~\cite{wagner2010maxima}. 

A natural and recent extension of independent sets is the notion of \emph{$k$-nearly independent vertex subsets}, where a subset of vertices induces exactly $k$ edges. This generalisation provides more tools for the analysis of graph structure. The systematic study of $k$-nearly independent subsets was recently advanced (see~\cite{andriantiana2024number}) by two of the authors of this article, who introduced and investigated the parameter $\sigma_1(G)$, the number of $1$-nearly independent vertex subsets, establishing foundational recursive relations and extremal bounds. Further developments were made in~\cite{DossouOloryAndriantiana2025}, exploring the average size of such subsets and highlighting their combinatorial significance. This was followed by the Nordhaus-Gaddum-type results in \cite{andriantiana2026nordhaus}.

Beyond enumeration, comparative studies of different invariants have emerged as a powerful tool for understanding the interplay between different structural parameters. Classical examples include correlations (see~\cite{Wagner2007Correlation}) between parameters and a ratio~\cite{WienerRatio2023}. Such ratios often reveal more refined structural information than the single invariants themselves, and they have been used to characterise extremal graphs, compare graph classes, and study asymptotic behaviour. Motivated by these developments, we consider in this paper the ratio
\[
Q(G) = \frac{\sigma_1(G)}{\sigma_0(G)}\,.
\]
%This ratio, which we refer to as the $(1,0)$-ratio or simply the \emph{Stellenbosch ratio}, provides a natural measure of how the allowance of a single edge in a vertex subset increases the combinatorial richness of the graph. Understanding the extremal behaviour of this ratio offers new insights into the structure of graphs that maximise or minimise near-independence.

We establish sharp lower and upper bounds, as well as identifying extremal graphs under various structural constraints. Our approach combines recursive techniques, decomposition methods, and inequalities derived from classical results on $\sigma_0(G)$ and $\sigma_1(G)$.

The paper is organised as follows. In Section~\ref{Sec:PRE}, we start our preparation with preliminary results, including key recursive relations for $\sigma_0(G)$ and $\sigma_1(G)$, as well as known extremal properties that will be used throughout the paper. In Section~\ref{Sec:LOW}, we provide lower bounds for the ratio $Q(G)$, first for connected graphs and then for more general graph classes, with particular attention to the role of maximum degree and graph decomposition. Section~\ref{Sec:UPP} is devoted to upper bounds, where we derive sharp inequalities for forests and trees using inductive techniques and refined structural lemmas. These results are further improved through alternative formulations that yield tighter bounds.

\section{Preliminary Results}
\label{Sec:PRE}
In this section, we present some preliminary results that will be helpful in proving our main results. 

Let $G=(V(G), E(G))$ be a graph, $v$ a vertex of $G$, and $S$ a set of vertices or edges of $G$. We denote by $G-S$ the subgraph obtained by removing from $G$ all elements of $S$ along with their incident edges.  We simply write $G-v$ instead of $G-\{v\}$. %Note that $G-S_1-S_2=G-(S_1 \cup S_2)$. 

We represent by $N_G(v)=\{u: uv\in E(G)\}$ the (open) neighbourhood of $v$ in $G$ and by $N_G[v]=N_G(v)\cup\{v\}$ the closed neighbourhood of $v$ in $G$. As usual, the star and the path with $n$ vertices are denoted by $S_n$ and $P_n$ respectively. For two vertex disjoint graphs $G_1$ and $G_2$, we denote by $G_1 \cup G_2$ the graph with the vertex set $V(G_1) \cup V(G_2)$ and edge set $E(G_1) \cup E(G_2)$, while we simply write $m G$ instead of $G \cup G \cup \cdots \cup G$ ($m$ copies of $G$). By $\overline{G}$, we mean the complement of $G$, i.e. the graph with vertex and edge sets $V(G)$ and $\{uv:~ uv\notin E(G)\}$, respectively. We write $\overline{K_n}$ for $nP_1$. The number of vertices (resp. edges) of $G$ is referred to as the order (resp. size) of $G$.

\medskip

The following well-known recursive relation will be used often.
\begin{lem}[\cite{andriantiana2024number, merrifield1981enumeration}]
\label{Lem:Rec}
For any vertex $v$ in a graph $G$ we have
$$
\sigma_0(G)= \sigma_0(G-v)+\sigma_0(G-N_G[v])\,,$$
and 
$$
\sigma_1(G)= \sigma_1(G-v)+\sigma_1(G-N_G[v])+\sum_{u\in N_G(v)} \sigma_0(G-N_G[v]-N_G[u]).
$$
\end{lem}

\medskip
The next two results stated here as lemmas can be found in~\cite{andriantiana2024number,wagner2010maxima}.

\begin{lem}[\cite{andriantiana2024number}]
    \label{lem:lower-bound-for-sigma-1-of-connected-graphs}
     If $G$ is a connected graph of order $n$, then
     \begin{align*}
         \sigma_1(G) \ge \sigma_1(S_n) = n-1,
     \end{align*}
     with equality if and only if $G = S_n$.
\end{lem}

\begin{lem}[\cite{wagner2010maxima}]
    \label{lem:upper-bound-for-sigma-0-of-connected-graphs}
     If $G$ is a connected graph of order $n$, then
     \begin{align*}
         \sigma_0(G) \le \sigma_0(S_n) = 2^{n-1}+1,
     \end{align*}
     with equality if and only if $G = S_n$.
\end{lem}

\begin{lem}[\cite{wagner2010maxima}]
     \label{lem:sigma_0-and-Z_0}
     If $G=\displaystyle \bigcup\limits_{i=1}^r G_i$, then 
     \begin{align*}
         \sigma_0(G) = \sigma_0\left(\bigcup\limits_{i=1}^r G_i\right) = \prod\limits_{i=1}^r \sigma_0(G_i). %\quad \text{ and } \quad Z_0(G) = Z_0\left(\bigcup\limits_{i=1}^r G_i\right) = \prod\limits_{i=1}^r Z_0(G_i).
     \end{align*}
 \end{lem}

 \begin{lem}[\cite{andriantiana2024number}]
     \label{lem:sigma_1-and-Z_1}
     If $G = G_1 \cup G_2$, then 
     \begin{align*}
         \sigma_1(G) = \sigma_1(G_1)\sigma_0(G_2) + \sigma_1(G_2)\sigma_0(G_1).
     \end{align*}
 \end{lem}

%\subsection{Relationships between ratios}

%In this subsection, we present a relationship between the $(1,0)$-ratio for $\overline{\sigma}_1(G)$ and the $(1,0)$-ratio for $\sigma_1(G)$ as well as a relationship between the $(1,0)$-ratio for $\overline{Z}_1(G)$ and the $(1,0)$-ratio for $Z_1(G)$ if $G$ is any graph. Note that for any graph $G$, we have $\overline{\sigma}_0(G)= \sigma_0(G)$ and $\overline{Z}_0(G) = Z_0(G)$.

% \begin{lem}
%     \label{lem:ratio-for-cum-sigma-1}
%     For any graph $G$, we have
%     \begin{itemize}
%         \item[(i)] ${\displaystyle \frac{\overline{\sigma}_{1}(G)}{\overline{\sigma}_{0}(G)} = 1 + \frac{\sigma_1(G)}{\sigma_0(G)},}$
%         \item[(ii)] ${\displaystyle \frac{\overline{Z}_{1}(G)}{\overline{Z}_{0}(G)} = 1 + \frac{Z_1(G)}{Z_0(G)}.}$
%     \end{itemize}
% \end{lem}

% \begin{proof}
%     Let $G$ be a graph. Then,
%     \begin{align*}
%         \frac{\overline{\sigma}_{1}(G)}{\overline{\sigma}_{0}(G)} = \frac{\sigma_0(G) + \sigma_1(G)}{\sigma_0(G)} = \frac{\sigma_0(G)}{\sigma_0(G)} + \frac{\sigma_1(G)}{\sigma_0(G)} = 1 + \frac{\sigma_1(G)}{\sigma_0(G)}, 
%     \end{align*}
%     as claimed in (i). Similarly,
%     \begin{align*}
%         \frac{\overline{Z}_{1}(G)}{\overline{Z}_{0}(G)} = \frac{Z_0(G) + Z_1(G)}{Z_0(G)} = \frac{Z_0(G)}{Z_0(G)} + \frac{Z_1(G)}{Z_0(G)} = 1 + \frac{Z_1(G)}{Z_0(G)}, 
%     \end{align*}
%     as claimed in (ii).
% \end{proof}

\medskip
Lemmas~\ref{lem:sigma_0-and-Z_0} and~\ref{lem:sigma_1-and-Z_1} lead to the following lemma on the ratio $\sigma_1/\sigma_0$. 

\begin{lem}[\cite{DossouOloryAndriantiana2025}]
\label{Lem:AVSum}
Let $G_1$ and $G_2$ be two vertex disjoint graphs. We have
\begin{align*}
\frac{ \sigma_1(G_1 \cup G_2)}{ \sigma_0(G_1 \cup G_2)}=\frac{ \sigma_1(G_1)}{ \sigma_0(G_1)}+\frac{ \sigma_1(G_2)}{ \sigma_0(G_2)}\,.
\end{align*}
\end{lem}

\medskip
In the next section, we present several lower bounds for $Q$, beginning with straightforward ones and progressing to more interesting ones.

\section{Lower bounds}
\label{Sec:LOW}

The following theorem is trivial, we skip the proof.
\begin{thm}
    \label{thm:lower-bound-for-sigma-ratio-general-graphs}
    If $G$ is a general graph of order $n$, then 
    \begin{align*}
        \frac{\sigma_1(G)}{\sigma_0(G)} \ge \frac{\sigma_1(\overline{K_n})}{\sigma_0(\overline{K_n})} = \frac{0}{2^n} =0\,,
    \end{align*}
    with equality if and only if $G = \overline{K_n}$.
\end{thm}

\medskip
One might think from Theorem~\ref{thm:lower-bound-for-sigma-ratio-general-graphs} that removing an edge from a graph 
G can possibly reduce the ratio $\sigma_1/\sigma_0$, however this is not always the case. In fact if we consider the path $P_4$ on $4$ vertices, and $uv$ its middle edge, then $P_4-uv=2K_2$ and
$$
\frac{\sigma_1(P_4)}{\sigma_0(P_4)}
=\frac{5}{8}<\frac{6}{9}=\frac{\sigma_1(2K_2)}{\sigma_0(2K_2)}\,.
$$

\medskip
We now aim to determine the connected graph of order $n$ that has the smallest $Q$.%=\sigma_1/\sigma_0$.
\begin{thm}
    \label{thm:lower-bound-for-sigma-ratio-of-connected-graphs}
    If $G$ is a connected graph of order $n$, then 
    \begin{align*}
        \frac{\sigma_1(G)}{\sigma_0(G)} \ge \frac{\sigma_1(S_n)}{\sigma_0(S_n)} = \frac{n-1}{2^{n-1}+1}\,,
    \end{align*}
    with equality if and only if $G =S_n$.
\end{thm}

\begin{proof}
    Let $G$ be any connected graph of order $n$. By Lemma~\ref{lem:lower-bound-for-sigma-1-of-connected-graphs} we have $\sigma_1(G) \ge \sigma_1(S_n)$, and by Lemma~\ref{lem:upper-bound-for-sigma-0-of-connected-graphs} we have $\sigma_0(S_n) \ge \sigma_0(G)$. Thus,
    \begin{align*}
        \frac{\sigma_1(G)}{\sigma_0(G)} \ge \frac{\sigma_1(S_n)}{\sigma_0(S_n)}=\frac{n-1}{2^{n-1}+1}\,.
    \end{align*}
To prove the characterisation of the equality case, suppose that $G$ is not %isomorphic to 
$S_n$ and that $G$ satisfies 
    \begin{align*}
        \frac{\sigma_1(G)}{\sigma_0(G)} = \frac{n-1}{2^{n-1}+1}.
    \end{align*}
    By Lemma~\ref{lem:lower-bound-for-sigma-1-of-connected-graphs}, $\sigma_1(G) \ge n$ and by Lemma~\ref{lem:upper-bound-for-sigma-0-of-connected-graphs}, $2^{n-1}\geq \sigma_0(G)$. Thus,
    \begin{align*}
     \frac{\sigma_1(G)}{\sigma_0(G)} \ge \frac{n}{2^{n-1}}\,.
    \end{align*}
    However,
    \begin{align*}
        \frac{n}{2^{n-1}} - \frac{n-1}{2^{n-1}+1} &= \left( \frac{n}{2^{n-1}}  +   \frac{1}{2^{n-1}+1}    \right) -  \frac{n}{2^{n-1}+1}\\
        &>\left( \frac{n}{2^{n-1}+1}  +   \frac{1}{2^{n-1}+1}    \right) -  \frac{n}{2^{n-1}+1}\\
        &>0,
    \end{align*}
    a contradiction. This completes the proof of Theorem~\ref{thm:lower-bound-for-sigma-ratio-of-connected-graphs}.
\end{proof}

\medskip
Since $S_n$ is a tree, we immediately have the following corollary.

\begin{cor}
    \label{cor:lower-bound-for-sigma-ratio-of-trees}
    If $T$ is a tree of order $n$, then 
    \begin{align*}
        \frac{\sigma_1(T)}{\sigma_0(T)} \ge \frac{\sigma_1(S_n)}{\sigma_0(S_n)}\,,
    \end{align*}
    with equality if and only if $T = S_n$.
\end{cor}

\medskip
Our next result extends Theorem~\ref{thm:lower-bound-for-sigma-ratio-of-connected-graphs} to general (possibly disconnected) graphs. Before we get to it, we need to separately study the case where the maximum degree is $1$.

\begin{prop}
\label{thm:lower-bound-for-sigma-ratio-of-general-graphs-of-max-degree-1}
    If $G$ is a graph with maximum degree $\Delta(G) =1$, then 
    \begin{align*}
        \frac{\sigma_1(G)}{\sigma_0(G)} \ge \frac{1}{3},
    \end{align*}
    with equality if and only if $G = K_2 \cup \overline{K_{n-2}}$.
\end{prop}

\begin{proof}
    Let $G$ be a graph with maximum degree $\Delta(G) =1$. Then $G \cong (tK_2) \cup \overline{K_{n-2t}}$ for some integer $t\ge 1$. Thus, by Lemmas~\ref{lem:sigma_0-and-Z_0} and~\ref{lem:sigma_1-and-Z_1}, we have
    \begin{align*}
        \sigma_1(G) = \sigma_1(tK_2 \cup \overline{K_{n-2t}}) &= \sigma_1(t K_2)\sigma_0(\overline{K_{n-2t}}) + \sigma_0(tK_2) \sigma_1(\overline{K_{n-2t}})\\
        &=t(3)^{t-1}\cdot 2^{n-2t} + 0\cdot 3^t =t(3)^{t-1}\cdot 2^{n-2t}\,,
    \end{align*}
   and also
    \begin{align*}
        \sigma_0(G) = \sigma_0(tK_2 \cup \overline{K_{n-2t}}) = \sigma_0(tK_2) \sigma_0(\overline{K_{n-2t}}) = 3^t \cdot 2^{n-2t}\,.
    \end{align*}
    Hence, 
        \begin{align*}
        \frac{\sigma_1(G)}{\sigma_0(G)} = \frac{\sigma_1(tK_2 \cup \overline{K_{n-2t}})}{\sigma_0(tK_2 \cup \overline{K_{n-2t}})} = \frac{t\cdot 3^{t-1}\cdot 2^{n-2t}}{3^t \cdot 2^{n-2t}} = \frac{t}{3} \geq  \frac{1}{3}
    \end{align*}
 for all $t\ge 1$. Equality holds if and only if $t=1$, that is, $G = K_2 \cup \overline{K_{n-2}}$.
\end{proof}

\medskip
We treat empty sums as $0$. Following Theorem~\ref{thm:lower-bound-for-sigma-ratio-general-graphs}, we now obtain the $n$-vertex graph with the second-smallest value of $Q$.

\begin{thm}
\label{Thm:DiscMin} Let $n\geq 4$ and $G\neq \overline{K_n}$ be a graph of order $n$. It holds that
$$
\frac{\sigma_1(G)}{\sigma_0(G)}\geq \frac{\sigma_1(S_n)}{\sigma_0(S_n)}\,.
$$
\end{thm}

\begin{proof}
Define $f(j)=(j-1)/(2^{j-1}+1)=Q(S_{j})$. Denote by $c_2$ the number of connected components of $G$ that have $2$ vertices, and let $G_1,\dots,G_{c_3}$ be the connected components of $G$ that have more than $2$ vertices, with $n_1,\dots,n_{c_3}$ their respective orders. Note that isolated vertices do not affect the value of $Q$, meaning that $Q(H)=Q(H-w)$ if $w$ is an isolated vertex in a graph $H$.

By Lemma~\ref{Lem:AVSum} alongside Theorems~\ref{thm:lower-bound-for-sigma-ratio-of-general-graphs-of-max-degree-1} and \ref{thm:lower-bound-for-sigma-ratio-of-connected-graphs}, we have
\begin{align}
\label{Eq:Dis}
\frac{\sigma_1(G)}{\sigma_0(G)}
&=\frac{\sigma_1 (c_2K_2)}{\sigma_0(c_2K_2)}+\sum_{i=1}^{c_3}\frac{\sigma_1(G_i)}{\sigma_0(G_i)}\geq c_2\frac{\sigma_1(K_2)}{\sigma_0(K_2)}+\sum_{i=1}^{c_3}\frac{\sigma_1(S_{n_i})}{\sigma_0(S_{n_i})}
=c_2f(2)+\sum_{i=1}^{c_3}f(n_i)\,.
\end{align}
Note that
$$f(1)=0<f(2)=\frac13<f(3)=\frac{2}{5}> f(4)=\frac{3}{9}=f(2)$$ and that
 $$f'(x)=\frac{2^{x-1}+1 -\ln(2)2^{x-1}(x-1)}{(2^{x-1}+1)^2}=\frac{2^{x-1}(1-\ln(2)(x-1))+1}{(2^{x-1}+1)^2}<0 \text{ for all real }x\geq 4\,,$$
 as $1-\ln(2)(x-1) <-1$ for $x\geq 4$. So $f(i)$, restricted to integers $i$, decreases from $i=3$.

If $c_3\neq 0$, then by~\eqref{Eq:Dis}, 
 $$
  \frac{\sigma_1(G)}{\sigma_0(G)}\geq \sum_{i=1}^{c_3}f(n_i)\geq f\left(n_1\right) \geq f(n)=\frac{\sigma_1(S_n)}{\sigma_0(S_n)}\,.
 $$

Otherwise, $c_3=0$. With $G\neq \overline{K_n}$, the graph $G$ is of maximum degree $1$. By Proposition~\ref{thm:lower-bound-for-sigma-ratio-of-general-graphs-of-max-degree-1}, we have
 $$
  \frac{\sigma_1(G)}{\sigma_0(G)}\geq \frac{1}{3}=f(4)\geq f(n)=\frac{\sigma_1(S_n)}{\sigma_0(S_n)}\,,
 $$ which completes the proof of the theorem.
\end{proof}

\medskip
It is striking that the $n$-vertex graphs with smallest value of $Q$ has maximum degree $0$, while the one with the second-smallest has maximum degree $n-1$.

\medskip
To end this section, we provide a generalisation of Proposition~\ref{thm:lower-bound-for-sigma-ratio-of-general-graphs-of-max-degree-1} for arbitrary maximum degree $\Delta$.

\begin{thm}
If $G$ is a graph with maximum degree $\Delta$, then
$$
\frac{\sigma_1(G)}{\sigma_0(G)}\geq \min \left\lbrace \frac{1}{3},\frac{\sigma_1(S_{\Delta+1})}{\sigma_0(S_{\Delta+1})}\right\rbrace \,.
$$
Equality happens if and only if $G$ is  $S_{\Delta+1}\cup \overline{K_{|V(G)|-\Delta -1}}$.
\end{thm}

\begin{proof}
We proceed by induction on the order $n\geq \Delta +1$ of $G$. For $n=\Delta+1$, $G$ is connected. We apply Theorem~\ref{thm:lower-bound-for-sigma-ratio-of-connected-graphs} and obtain
$$
\frac{\sigma_1(G)}{\sigma_0(G)}\geq \frac{\sigma_1(S_{n})}{\sigma_0(S_n)}\,,
$$
as desired. Suppose that the statement holds whenever $\Delta +1\leq n <k$ for some integer $k$. Now assume that $n=k$. By Lemma~\ref{Lem:AVSum}, we can take the ratio $Q$ of one connected component of $G$ as a lower bound. Hence, we can assume that $G$ is connected.

Let $v$ be a vertex of degree $\Delta$ in $G$. If $G-v$ has no edge, then $G$ is a star and the theorem holds trivially. Otherwise, if an edge is also present in $G-N_G[v]$, then we have
\begin{align*}
\frac{\sigma_1(G)}{\sigma_0(G)}
&=\frac{\sigma_1(G-v)+\sigma_1(G-N_G[v])+ \sum_{x\in N_G(v)}\sigma_0(G-N_G[x]-N_G[v])}{\sigma_0(G-v)+\sigma_0(G-N_G[v])}\\
&>\frac{\sigma_1(G-v)+\sigma_1(G-N_G[v])}{\sigma_0(G-v)+\sigma_0(G-N_G[v])}=\rho\frac{\sigma_1(G-v)}{\sigma_0(G-v)}+(1-\rho)\frac{\sigma_1(G-N_G[v])}{\sigma_0(G-N_G[v])}
\end{align*}
with $\rho= \sigma_0(G-v) / (\sigma_0(G-v)+\sigma_0(G-N_G[v]))$. Let $\Delta_1$ an $\Delta_2$ be the maximum degrees of $G-v$ and $G-N_G[v]$, respectively. The induction hypothesis implies that
\begin{align*}
\frac{\sigma_1(G)}{\sigma_0(G)}
& > \min\left\lbrace\frac{1}{3},\frac{\sigma_1(S_{\Delta_1 +1})}{\sigma_0(S_{\Delta_1 +1})}, \frac{\sigma_1(S_{\Delta_2 +1})}{\sigma_0(S_{\Delta_2 +1})}\right\rbrace\geq \min\left\lbrace\frac{1}{3},\frac{\sigma_1(S_{\Delta+1})}{\sigma_0(S_{\Delta+1})}\right\rbrace\,,
\end{align*}
where the last inequality follows from the fact that $f(j)=(j-1)/(2^{j-1}+1)=Q(S_{j})$, restricted to integers $j$, decreases from $j=3$ (see the proof of Theorem~\ref{Thm:DiscMin}).

\medskip
Now, we assume that $G-N_G[v]$ has no edge. Since
 $$|V(G-N_G[x]-N_G[v])|\geq n- 2\Delta $$ by the choice of vertex $v$, we have
\begin{align*}
\frac{\sigma_1(G)}{\sigma_0(G)}
&=\frac{\sigma_1(G-v)+ \sum_{x\in N_G(v)}\sigma_0(G-N_G[x]-N_G[v])}{\sigma_0(G-v)+\sigma_0(G-N_G[v])}\\
&\geq \frac{\sigma_1(G-v)+\Delta 2^{n-2\Delta}}{\sigma_0(G-v)+2^{n-(\Delta +1)}}=\rho\frac{\sigma_1(G-v)}{\sigma_0(G-v)}+(1-\rho)\frac{\Delta 2^{n-2\Delta}}{2^{n-(\Delta +1)}}\\
&=\rho\frac{\sigma_1(G-v)}{\sigma_0(G-v)}+(1-\rho)\frac{\Delta}{2^{\Delta -1}}
>\rho\frac{\sigma_1(G-v)}{\sigma_0(G-v)}+(1-\rho)\frac{\Delta}{2^{\Delta}+1}
\end{align*}
with $\rho=\sigma_0(G-v) / (\sigma_0(G-v)+2^{n-(\Delta +1)} )$. Thus,
\begin{align*}
\frac{\sigma_1(G)}{\sigma_0(G)} > \rho\frac{\sigma_1(G-v)}{\sigma_0(G-v)}+(1-\rho)\frac{\sigma_1(S_{\Delta +1 })}{\sigma_0(S_{\Delta +1 })}.
\end{align*}
With $\Delta_1$ being the maximum degree in $G-v$ (so $\Delta_1\leq \Delta$), using the induction assumption, we get
\begin{align*}
\frac{\sigma_1(G)}{\sigma_0(G)}
&>\rho\frac{\sigma_1(G-v)}{\sigma_0(G-v)}+(1-\rho)\frac{\sigma_1(S_{\Delta +1 })}{\sigma_0(S_{\Delta +1 })}\geq \min\left\lbrace \frac{1}{3},\frac{\sigma_1(S_{\Delta +1 })}{\sigma_0(S_{\Delta +1 })} \right\rbrace \,.
\end{align*}
Our arguments also show that the case of equality can occur only if $G-v$ has no edge, i.e. $G$ is a star on $1+\Delta(G)$ vertices. Since isolated vertices do not affect the value of $Q$, the equality case is proved.
\end{proof}

\medskip
In the next section, we focus on finding upper bounds on $Q$.

\section{Upper bounds}
\label{Sec:UPP}

%For a graph $G$, we define 
%$$ Q(G)=\frac{\sigma_1(G)}{\sigma_0(G)}\,.$$

% \medskip
% The following result can be found in the recent preprint \cite{dossou-olory2025average}.
% \begin{prop}
% Let $G_1$ and $G_2$ be two vertex disjoint graphs. We have
% \begin{align*}
% \frac{ \sigma_1(G_1 \cup G_2)}{ \sigma_0(G_1 \cup G_2)}=\frac{ \sigma_1(G_1)}{ \sigma_0(G_1)}+\frac{ \sigma_1(G_2)}{ \sigma_0(G_2)}\,.
% \end{align*}
%Moreover, if $G=(V,E)$ has no isolated vertex and maximum degree at most $2$, then
%\begin{align*}
%\frac{\sigma_1(G)}{\sigma_0(G)}\geq \frac{1}{3}\,.
%\end{align*}
%\end{prop}
%We know that equality holds in the above proposition, see~Theorem~\ref{thm:lower-bound-for-sigma-ratio-of-general-graphs-of-max-degree-1}.

We now aim to prove the following sharp upper bound.
	\begin{thm}\label{MaxForest}
		Let $F$ be a forest of order $n\geq 1$. We have
        \begin{equation*}
			Q(F)\leq \frac{1}{3}\left(n-1\right)\,.
		\end{equation*}
The equality is reached for $F\in \{P_1,P_2\}$.
	\end{thm}
    
\medskip
    Before we prove the theorem, let us show two lemmas.
    
	\begin{lem}\label{Lem:1}
		If $T$ is a tree of order $n\geq 2$  and $v$ a leaf of $T$, then it holds that
		\begin{equation*}
			\frac{\sigma_0\left(T-N_T[v]\right)}{\sigma_0\left(T-v\right)}\leq 1-\frac{1}{2^{n-2}+1}\,.
		\end{equation*}
	\end{lem}
    
	\begin{proof}
		Let $v$ be a leaf attached to $u$ in an $n$-vertex tree $T$. We have
		\begin{equation}
			\frac{\sigma_0\left(T-N_T[v]\right)}{\sigma_0\left(T-v\right)} =1-\frac{\sigma_0\left(T-v\right)-\sigma_0\left(T-N_T[v]\right)}{\sigma_0\left(T-v\right)}\,. \label{eq:01}
		\end{equation}
		Moreover, $\sigma_0(T-v)=\sigma_0\left(T-v-u\right) + \sigma_0\left(T-v-N_T[u]\right)$. %, with $u$ being the vertex adjacent to $v$ in $T$. 
        Thus
		\begin{equation*}
			\sigma_0(T-v)-\sigma_0\left(T-N_T[v]\right)=\sigma_0\left(T-v-N_T[u]\right) = \sigma_0\left(T-N_T[u]\right)\geq 1\,.
		\end{equation*}
		Identity~\eqref{eq:01} implies that $$\frac{\sigma_0\left(T-N_T[v]\right)}{\sigma_0\left(T-v\right)} \leq 1-\frac{1}{\sigma_0\left(T-v\right)}\,.$$
		Since $T-v$ is a tree,
        $$\sigma_0\left(T-v\right)\leq \sigma_0\left(S_{n-1}\right)=1+2^{n-2}\,.$$ This proves the lemma.
	\end{proof}

    \medskip
	\begin{lem}\label{Lem:2}
		If $T$ is a tree and $v$ a leaf of $T$, then it holds that
		\begin{equation*}
			Q(T)\leq \frac{\displaystyle\frac{\sigma_0\left(T-v\right)}{\sigma_0\left(T-N_T[v]\right)}Q(T-v)+1+Q\left(T-N_T[v]\right)}{1+\displaystyle\frac{\sigma_0\left(T-v\right)}{\sigma_0\left(T-N_T[v]\right)}}\,.
		\end{equation*}
	\end{lem}
    
	\begin{proof}
		From 
        $$\sigma_1(T)=\sigma_1\left(T-v\right)+\sigma_1\left(T-N_T[v]\right)+\sigma_0\left(T-N_T[u]\right)\,,$$
		where $u$ is the vertex adjacent to $v$ in $T$, together with
		$$\sigma_0(T)=\sigma_0\left(T-v\right)+\sigma_0\left(T-N_T[v]\right)\,,$$ we have
		\begin{equation*}
			Q(T)=\frac{\sigma_1(T)}{\sigma_0(T)}=\frac{\sigma_0(T-v)Q(T-v)+\sigma_0\left(T-N_T[v]\right)Q\left(T-N_T[v]\right)+\sigma_0\left(T-N_T[u]\right)}{\sigma_0(T-v)+\sigma_0\left(T-N_T[v]\right)}.
		\end{equation*}
		The graph $T-N_T[u]$ is a subgraph of $T-N_T[v]$, so $\sigma_0\left(T-N_T[u]\right)\leq \sigma_0\left(T-N_T[v]\right)$. This shows that
        \begin{equation*}
			Q(T) \leq \frac{\sigma_0(T-v)Q(T-v)+\sigma_0\left(T-N_T[v]\right)Q\left(T-N_T[v]\right)+\sigma_0\left(T-N_T[v]\right)}{\sigma_0(T-v)+\sigma_0\left(T-N_T[v]\right)}\,,
		\end{equation*}
        thus proving the lemma, upon dividing through by $\sigma_0\left(T-N_T[v]\right)$.	
	\end{proof}

    \medskip
    We can now prove Theorem~\ref{MaxForest}.
    
	\begin{proof} [Proof of Theorem~\ref{MaxForest}]
		We proceed by induction on the order of the forest $F$. Let $n=|F|:=|V(F)|$. If $n=1$, then $F=P_1$, $Q(F)=Q(P_1)=0$, and $0= \frac{1}{3}(1-1)$. If $n=2$, then $F=P_2$ or $F=P_1\cup P_1$. We have
			$Q(P_2)=\frac{1}{3}$, while $\frac{1}{3}\left(2-1\right)=\frac{1}{3}$. Also
			$Q(P_1\cup P_1)=0<\frac{1}{3}\left(2-1\right)$.
            
			Assume that the theorem holds for all forests with order less than or equal to $n-1$, for some $n\geq 3$. Let $F$ be a forest of order $n$. If $F$ is not connected , then $F=\cup_{i=1}^mT_i$, for some $m\geq 2$. In this case,
			\begin{equation*}
			Q(F)=\sum_{i=1}^{m}Q(T_i)\leq \sum_{i=1}^{m}\frac{1}{3}\left(|T_i|-1\right)
					=\frac{1}{3}|T|-\frac{1}{3}m < \frac{1}{3}|T|-\frac{1}{3}\,.
			\end{equation*}
			So, assume $F$ is a tree $T$. 
    %\medskip
	%We continue the proof of the theorem.
    Set $a:=\frac{1}{3}$. By Lemma~\ref{Lem:2} together with the induction hypothesis, we get
		\begin{equation*}
	\begin{aligned}
			Q(T)&\leq \frac{\displaystyle\frac{\sigma_0\left(T-v\right)}{\sigma_0\left(T-N_T[v]\right)}Q(T-v)+1+Q\left(T-N_T[v]\right)}{1+\displaystyle\frac{\sigma_0\left(T-v\right)}{\sigma_0\left(T-N_T[v]\right)}}\\
			&\leq \frac{\displaystyle\frac{\sigma_0\left(T-v\right)}{\sigma_0\left(T-N_T[v]\right)}(a(n-2))+1+a(n-3)}{1+\displaystyle\frac{\sigma_0\left(T-v\right)}{\sigma_0\left(T-N_T[v]\right)}}\,.
	\end{aligned}
	\end{equation*}
    
	Using Lemma~\ref{Lem:1}, that is $$\frac{\sigma_0\left(T-v\right)}{\sigma_0\left(T-N_T[v]\right)}\geq \frac{1}{\displaystyle 1-\frac{1}{2^{n-2}+1}}=\frac{2^{n-2}+1}{2^{n-2}}\,,$$ 
    together with the fact that the function $x \mapsto (Ax+B)/(1+x)$ is decreasing  if $A< B$, we obtain
	\begin{equation*}
		Q(T)\leq \frac{\displaystyle\frac{2^{n-2}+1}{2^{n-2}}\left(a(n-2)\right)+ 1+a(n-3)}{1+\displaystyle \frac{2^{n-2}+1}{2^{n-2}}}\,.
	\end{equation*}
	Set $c_n:=(2^{n-2}+1) / 2^{n-2}$. We have
	\begin{equation*}
		Q(T)\leq \frac{c_n\left(an-a-a\right)+an-2a-a+1}{1+c_n}.
	\end{equation*}
  We finish the proof by showing that  
$$c_n\left(an-a-a\right)+an-2a-a+1\leq \left(1+c_n\right)\left(an-a\right)\,.$$
This simplifies to 
%This is equivalent to showing that
	\begin{equation*}
		\begin{aligned}
			%& c_n(-a)+an-2a-a+1\leq an-a\,,\\
			& ac_n+2a-1\geq 0\,.
		\end{aligned}
	\end{equation*}
	Since $c_n\geq 1$ for all positive integers $n$, it suffices to show that $a(1)+2a-1 \geq 0$, that is $a\geq \frac{1}{3}$,
	which is indeed true.
	\end{proof}

\medskip
Now we improve on the upper bound, by using a different approach.

\begin{lem}\label{Lem:SimpleQ}
Let $v$ be a leaf of a tree $T$ attached to a vertex $u$. We have
\begin{align*}
  \sigma_0(T)&=2\sigma_0(T-N_T[v])+\sigma_0(T-N_T[u])\,,\\
    Q(T)&=\frac{2\sigma_0(T-N_T[v])}{\sigma_0(T)}\cdot \frac{1}{2} (Q(T-v)+Q(T-N_T[v])) + \frac{\sigma_0(T-N_T[u])}{\sigma_0(T)}(1+Q(T-v))\,.
\end{align*}
\end{lem}

\begin{proof}
By Lemma~\ref{Lem:Rec}, we have $\sigma_0(T)=\sigma_0(T-v)+\sigma_0(T-N_T[v])$ and 
\begin{align*}
    \sigma_0(T-v)=\sigma_0((T-v)-u)+\sigma_0((T-v)-N_T[u])=\sigma_0(T-N_T[v])+\sigma_0(T-N_T[u])\,.
\end{align*}
Thus $\sigma_0(T)=2\sigma_0(T-N_T[v])+\sigma_0(T-N_T[u])$. Using Lemma~\ref{Lem:Rec} for $\sigma_1$, we obtain
\begin{align*}
    \sigma_1(T) &= \sigma_1(T-v)+ \sigma_1(T-N_T[v])+\sigma_0(T-N_T[u]) \\
    &=\sigma_0(T-v) Q(T-v) +\sigma_0(T-N_T[v])Q(T-N_T[v])+\sigma_0(T-N_T[u]) \\
      &= \big( \sigma_0(T-N_T[v])+\sigma_0(T-N_T[u]) \big) Q(T-v)\\
      & \hspace{6cm} +\sigma_0(T-N_T[v])Q(T-N_T[v])+\sigma_0(T-N_T[u]) \\
     &=\sigma_0(T-N_T[v]) \big(Q(T-v)+Q(T-N_T[v]) \big) +\sigma_0(T-N_T[u]) (1+ Q(T-v))\,.
\end{align*}
This completes the proof of the lemma.
\end{proof}

\medskip
Our final theorem reads as follows.

\begin{thm}
Let $F$ be a forest of order $n\geq 1$. We have the following:
\begin{align*}
    Q(F)\leq \frac{1}{4}n - \frac{1}{6}\,.
\end{align*}
\end{thm}

\begin{proof}
The theorem clearly holds for $n\leq 2$ (with equality for $n=2$). Consider a forest $F$ of order $n\geq3$. For simplicity, set $a:=1/4$ and $c:=1/6$. If $F$ is not connected, then let $T_1, T_2, \ldots, T_m$, for some integer $m\geq 2$, denote all its connected components, with orders $n_1,n_2,\ldots,n_m$, respectively. Since
$$Q(F)=Q(T_1)+Q(T_2)+\cdots + Q(T_m)\,,$$ we have
\begin{align*}
Q(F)\leq \sum_{j=1}^m (a\cdot n_j-c)= a \left(\sum_{j=1}^m n_i\right) - m \cdot c = a\cdot n - m\cdot c
\end{align*}
by the induction hypothesis. Since $m\geq 2$, 
$$ Q(F) \leq a\cdot n -2c <  a\cdot n - c\,.$$
If $F$ is connected, then we fix a leaf $v$ attached to a vertex $u$ in $F$, and invoke Lemma~\ref{Lem:SimpleQ} to obtain
\begin{align*}
    Q(F)=\rho \cdot \frac{1}{2} (Q(F-v)+Q(F-N_F[v])) + (1-\rho)(1+Q(F-v))
\end{align*}
with 
\begin{align*}
    \rho=\frac{2\sigma_0(F-N_F[v])}{2\sigma_0(F-N_F[v])+\sigma_0(F-N_F[u])}\,.
\end{align*}
The induction hypothesis yields
\begin{align*}
    Q(F)\leq & ~ \rho \cdot \frac{1}{2} \big( a(n-1)-c)  + a(n-2)-c \big) + (1-\rho)(1+ a (n-1) -c )\\
    &= \rho \big(a(n-1)-c -\frac{1}{2}a \big)  + (1-\rho)(1+ a (n-1) -c )\\
    &= a\cdot n - c + 1-a -\rho \left(1+\frac{1}{2}a\right)\,.
\end{align*}
Thus, it suffices to prove that 
$$1-a -\rho \left(1+\frac{1}{2}a \right) \leq 0\,,$$ or equivalently (using $a=1/4$),
$$ \frac{2d}{2d+1}=\rho \geq \frac{1-a}{1+a/2}=\frac{2-2a}{2+a}$$
with 
$$d=\frac{\sigma_0(F-N_F[v])}{\sigma_0(F-N_F[u])}\,.$$
In other words, 
$$d \geq \frac{1-a}{3a}=1\,,$$
which is indeed true, since $F-N_F[u]$ is a subgraph of $F-N_F[v]$. The inductive argument and thus the proof of the theorem is complete.
\end{proof}

\bigskip
%%%%%%%%%%%%%%%%%%%%%%%%BIBLIOGRAPHY%%%%%%%%%%%%%%%%%%%%%%%
%\bibliographystyle{abbrv}

%\bibliography{ratio-references}

\end{document}